\documentclass[letterpaper, 11pt]{amsart}
\usepackage{amsmath,amssymb,amsxtra,amsthm, amstext, amscd,amsfonts,fancyhdr, hyperref,enumerate, longtable, mathrsfs}
\usepackage{diagrams}
\usepackage{mleftright}

\hypersetup{
colorlinks=true,
urlcolor=black,
citecolor=blue,
linkcolor=blue,
}

\usepackage[OT2,T1]{fontenc}
\DeclareSymbolFont{cyrletters}{OT2}{wncyr}{m}{n}
\DeclareMathSymbol{\Sha}{\mathalpha}{cyrletters}{"58}

\newcommand{\bF}{{\mathbb{F}}}

\newcommand{\bQ}{{\mathbb{Q}}}

\newcommand{\bZ}{{\mathbb{Z}}}

\newcommand{\Bx}{{\mathbf{x}}}

\newcommand{\BB}{\mathbf{B}}

\renewcommand{\L}{{\mathcal{L}}}
  
  \newcommand{\N}{{\mathcal{N}}}
\renewcommand{\O}{{\mathcal{O}}}

  \newcommand{\T}{{\mathcal{T}}}

\newcommand{\Gal}{\operatorname{Gal}}

\newcommand{\Cl}{\operatorname{Cl}}

\newcommand{\ep}{\varepsilon}

\newcommand{\Frob}{\operatorname{Frob}} 

\newcommand{\coker}{\operatorname{coker}}

\newcommand{\Li}{\operatorname{Li}}
\newcommand{\rk}{\operatorname{rk}} 
 
\newcommand{\Ind}{\operatorname{Ind}}

\newcommand{\ol}{\overline}

\newcommand{\upchi}{{\raise.35ex\hbox{$\chi$}}}

\makeatletter
\@namedef{subjclassname@2010}{%
  \textup{2010} Mathematics Subject Classification}
\makeatother

\newtheorem{theorem}{Theorem}[section]
\newtheorem{corollary}[theorem]{Corollary}
\newtheorem{proposition}[theorem]{Proposition}
\newtheorem{lemma}[theorem]{Lemma}

\theoremstyle{definition}

\theoremstyle{remark}

\numberwithin{equation}{section}

\allowdisplaybreaks

\begin{document}

\title{On the $\zeta_3$ Pell equation}


\author{Erick Knight}
\address{Department of Mathematics \\
University of Toronto \\
Bahen Centre \\
40 St. George Street, Room 6290 \\
Toronto, Ontario, Canada \\  M5S 2E4 }
\email{eknight@math.toronto.edu}
\indent

\author{Stanley Yao Xiao}
\address{Department of Mathematics \\
University of Toronto \\
Bahen Centre \\
40 St. George Street, Room 6290 \\
Toronto, Ontario, Canada \\  M5S 2E4 }
\email{syxiao@math.toronto.edu}
\indent


\begin{abstract} Let $K = \bQ(\zeta_3)$, where $\zeta_3$ is a primitive root of unity. In this paper we study the distribution of integers $\alpha \in \O_K$ for which the norm equation $N_{K(\sqrt[3]{\alpha})/K}(\Bx) = \zeta_3$ is solvable for integers $\Bx \in \mathcal{O}_{K(\sqrt[3]{\alpha})}$. The analogous question for $\zeta_2 = -1$ is the well-known negative Pell equation. We also address the natural generalization of Stevenhagen's conjecture on the negative Pell equation in this setting. 
\end{abstract}

\maketitle

\vspace{-5mm}

\section{Introduction}
A classical question in number theory is whether there is a solution to the Pell equation $x^2 - Dy^2 = 1$ in integers $x$ and $y$ for $D>1$ a square-free integer.  While this has an affirmitave answer for all $D$, the standard proof of this fact actually shows that there is a solution to $x^2-Dy^2 = \pm 1$ and then uses that to produce a solution to $x^2-Dy^2 = 1$.  Thus, one is led to ask questions about the negative Pell equation $x^2 -Dy^2 = -1$.  Many partial results have been obtained (e.g. \cite{FK2} and \cite{CKMP}).  The question about the negative Pell equation comes down determining whether there is a unit in the ring of integers of $\mathbb{Q}(\sqrt{D})$ which has norm $-1$.  Viewed from this perspective, it becomes possible to generalize this question to other classes of cyclic extensions.  In this paper, we will be focusing on the next simplest case.\\

Let $\zeta_3$ be a primitive third root of unity, and let $K = \mathbb{Q}(\zeta_3)$.  Then, by Kummer theory, all cyclic cubic extensions of $K$ are of the form $K(\sqrt[3]{\alpha})$ for $\alpha \in K^*$.  Motivated by the previous discussion, we are interested in whether there is a unit in $K(\sqrt[3]{\alpha})$ with norm $\zeta_3$.  One appealing aspect about this problem is that some of the structural issues which arise in the negative Pell equation problem goes away. On the other hand if one replaces $3$ with a larger prime, issues concerning global units would come up as the unit group of the base field will have positive rank. If the choice of prime was irregular, issues about relative class groups would arise.\\

Readers familiar with the negative Pell equation know that there are local obstructions that need to be satisfied for there to be a chance of a solution to the equation. In particular, every odd prime dividing $D$ must be congruent to $1$ mod $4$.  A similar issue arises here: one requires that every prime $\pi$ dividing $\alpha$ to have $N_{K/\mathbb{Q}}(\pi) \equiv 1 \pmod 9$ or have $\pi$ being an associate of $1-\zeta_3$, the prime above $3$.  In this introduction, we will restrict ourselves to the set of all $\alpha$ such that $K(\sqrt[3]{\alpha})/K$ is unramified at $1-\zeta_3$, as the case where the ramification at $1-\zeta_3$ is like $K_{1-\zeta_3}(\zeta_9)/K_{1-\zeta_3}$ provides no local obstruction to there being a solution to the $\zeta_3$-Pell equation but behaves very differently from the other cases. We will discuss this issue further in Section \ref{diffs}. \\

Put 
\begin{equation} S(X) = \{\alpha | N_{K/\mathbb{Q}}(D_{K(\sqrt[3]{\alpha})/K}) < X,  N_{K/\mathbb{Q}}(\pi) \equiv 1 \pmod{9} \forall \pi|\alpha \} \end{equation} 
and
\begin{equation} S^{\zeta_3}(X) = \{ \alpha \in S(X) | \exists u \in K(\sqrt[3]{\alpha} | N_{K(\sqrt[3]{\alpha})/K}(u) = \zeta_3\}. \end{equation}
We are interested in the size $\frac{|S^{\zeta_3}(X)|}{|S(X)|}$ as $X \rightarrow \infty$. 
To state the main theorem of the paper, let 
 \begin{equation} \label{beta} \beta = \prod_{i\geq 0} \left(1-\frac{1}{3^{2i+1}} \right) = \prod_{j=1}^\infty \left(1 + 3^{-j}\right)^{-1}. \end{equation}  
 This constant is the limit of the probability that a random large symmetric matrix over $\mathbb{F}_3$ has full rank.  Our main theorem is then the following:
\begin{theorem}\label{MainTheorem}
One has the lower bound
\[\beta \leq \liminf_{X \rightarrow \infty} \frac{|S^{\zeta_3}(X)|}{|S(X)|},\]
where $\beta$ is given as in (\ref{beta}). Further, one has the upper bound
\[\limsup_{X \rightarrow\infty} \frac{|S^{\zeta_3}(X)|}{|S(X)|} \leq \frac{3}{4}.\]
\end{theorem}
The proof of this theorem requires some applications of classical analytic estimates. The proof is in Section \ref{Proof}.  The version where one allows ramification at $1-\zeta_3$ is Theorem \ref{MainTheorem2}.
Following Stevenhagen, one might be tempted to make the conjecture that
\begin{equation} \label{FalseConjecture}\lim_{X\rightarrow\infty} \frac{|S^{\zeta_3}(X)|}{|S(X)|} = 2-2\beta. \end{equation}
We believe that (\ref{FalseConjecture}) is false, even when one assumes that there is no ramification at $1-\zeta_3$.  We have no theoretical explanation for why this is but there are numerical calculations that cast doubt on the underlying heuristic that produces this number. We discuss this issue in Section \ref{diffs}. \\

This paper is in three parts.  The first part will be algebraic criteria for having a solution or no solution to the $\zeta_3$-Pell equation.  The second part will be analysis to patch together the algebraic criteria in order to apply analytic methods.  The third part is then a discussion of why conjectures are different from the negative Pell case. \\

Finally, the astute reader may notice that this theorem is very similar to the one in \cite{FK1}.  Indeed, the overarching strategy is the same, and the lower bound is argued in a similar manner.  The upper bound is made complicated by the lack of an analogue of the narrow class group, and so we have to use techniques that were developed in \cite{Smith1} to get around this. However, te analytic estimates are much simpler since the Artin $L$-functions which arise in our setting cannot have Siegel zeroes. 


\section{Genus theory and R\'edei matricies}
This section will discuss genus theory and especially how it applies to the case of extensions of the form $K(\sqrt[3]{\alpha})/K$.  The main goal is to come up with algebraic criteria that characterize whether there is a unit $u \in \mathcal{O}_{K(\sqrt[3]{\alpha})}$ with norm equal to $\zeta_3$.

\subsection{The Cubic Residue Symbol and Cubic Reciprocity}
Let $\bZ[\zeta_3]$ be the ring of Eisenstein integers. This ring is a principal ideal domain and it has the units $\pm 1, \pm \zeta_3, \pm \zeta_6$. Up to units, the irreducible elements are $\lambda = 1 + \zeta_3$, the rational primes $q \equiv 2 \pmod{3}$, and the elements $\pi$ of $\bZ[\zeta_3]$, such that $\pi \ol{\pi}$ is a rational prime $p \equiv 1 \pmod{3}$. Denote by $\N$ the norm function on $\bZ[\zeta_3]$. \\

Let $\pi$ be an irreducible element in $\bZ[\zeta_3]$, and let $v \in \bZ[\zeta_3]$. We define the \emph{cubic} symbol $\left(\frac{v}{\pi}\right)_3$ by the formulas

\begin{equation} \left(\frac{v}{\pi}\right)_3 = \zeta_3^j, \pi \nmid v,
\end{equation}
where $j$ is the unique integer $j \in \{0,1,2\}$ such that 
\[v^{(N(\pi) - 1)/3} \equiv \zeta_3^j \pmod{\pi}\]
and
\[\left(\frac{v}{\pi}\right)_3 = 0 \text{ if } \pi | v.\]
Observe that if $\pi, \pi^\prime$ are associates, then $\left(\frac{\cdot}{\pi}\right)_3 = \left(\frac{\cdot}{\pi^\prime}\right)_3$ and the function $v \mapsto \left(\frac{v}{\pi}\right)_3$ is a multiplicative character of the group $(\bZ[\zeta_3]/\pi \bZ[\zeta_3])^\ast$. If $q$ is a rational prime congruent to $2$ mod $3$, the restriction to $\bZ$ of the corresponding cubic character is simply the principal character modulo $q$. We extend the definition of the cubic character to any element $w \in \bZ[\zeta_3]$ coprime to $3$, by the formula
\[\left(\frac{v}{w}\right)_3 = \prod_j \left(\frac{v}{w_j}\right)_3,\]
where $w = \prod_j w_j$ is the unique factorization (up to associates) of $w$ into irreducible elements.

\subsection{Genus theory}
Let $L/F$ be a cyclic extension of number fields, with $\Gal(L/F) = \langle \sigma \rangle$.  Write $I_L$ to be the group of fractional ideals of $L$ and $P_L$ to be the group of principal fractional ideals of $L$, and similarly for $F$.  There are two short exact sequences
$$ 0 \rightarrow \mathcal{O}_L^\times \rightarrow L^\times \rightarrow P_L \rightarrow 0\text{ and}$$
$$0 \rightarrow P_L \rightarrow I_L \rightarrow \Cl(L) \rightarrow 0.$$
Taking cohomology, and using the facts that $H^1(\langle\sigma\rangle, L^\times) = H^1(\langle\sigma\rangle, I_L) = 0$, we get three long exact sequences
$$ 0 \rightarrow \mathcal{O}_F^\times \rightarrow F^\times \rightarrow P_L[\sigma - 1] \rightarrow H^1(\langle\sigma\rangle, \mathcal{O}_L^\times) \rightarrow 0,$$
$$0 \rightarrow P_L[\sigma - 1] \rightarrow I_L[\sigma-1] \rightarrow \Cl(L)[\sigma-1] \rightarrow H^1(\langle\sigma\rangle, P_L) \rightarrow 0, \text{ and}$$
$$0 \rightarrow H^1(\langle\sigma\rangle, P_L) \rightarrow H^2(\langle\sigma\rangle, \mathcal{O}_L^\times) \rightarrow H^2(\langle\sigma\rangle, L^\times).$$
The image of $F^\times$ in $P_L[\sigma-1]$ is just $P_F$, so one gets $H^1(\langle\sigma\rangle, \mathcal{O}_L^\times) = P_L[\sigma-1]/P_F$.  Additionally, since $\langle\sigma\rangle$ is a cyclic group, one knows that $H^2(\langle\sigma\rangle, \mathcal{O}_L^\times) = \mathcal{O}_F^\times/N_{L/F}(\mathcal{O}_L^\times)$ and similarly for $H^2(\langle\sigma\rangle, L^\times)$.  Thus, one can replace $H^2(\langle\sigma\rangle, \mathcal{O}_L^\times) \rightarrow H^2(\langle\sigma\rangle, L^\times)$ with $H^2(\langle\sigma\rangle, \mathcal{O}_L^\times) \rightarrow \mathcal{O}_F^\times/(N_{L/F}(F^\times) \cap \mathcal{O}_F^\times) \rightarrow 0$.  Dividing the first two terms in the second sequence by $P_F$, we can stitch everything together into one sequence:
$$0 \rightarrow H^1(\langle\sigma\rangle, \mathcal{O}_L^\times) \rightarrow I_L[\sigma - 1]/P_F \rightarrow Cl(L)[\sigma - 1] \rightarrow H^2(\langle\sigma\rangle, \mathcal{O}_L^\times) \rightarrow \mathcal{O}_F^\times/(N_{L/F}(F^\times) \cap \mathcal{O}_F^\times) \rightarrow 0.$$

More detailed analysis is possible in general but at this point, we find it useful to switch to the case that $F = K (= \mathbb{Q}(\zeta_3))$ and $L/K$ is cyclic and degree $3$.  Kummer theory says that $L = K(\sqrt[3]{\alpha})$ for some $\alpha \in K$.  Every ideal of $K$ is principal, so the second term in the long exact sequence is $I_L[\sigma-1]/I_K$.  Now, we can write $I_L =  \bigoplus_{\mathfrak{P} \subset \mathcal{O}_K} \bigoplus_{\mathfrak{P}|\mathfrak{P}\mathcal{O}_L} \mathfrak{P}^{\mathbb{Z}}$.  If $\mathfrak{P}$ splits or is inert in $L$, then a straightforward calculation shows that $(\bigoplus_{\mathfrak{P}|\mathfrak{P}\mathcal{O}_L} \mathfrak{P}^{\mathbb{Z}})[\sigma-1] = \mathfrak{P}^\mathbb{Z}$, and if $\mathfrak{P}$ is ramified in $L$, then its easy to check that $(\bigoplus_{\mathfrak{P}|\mathfrak{P}\mathcal{O}_L} \mathfrak{P}^{\mathbb{Z}})[\sigma-1] = \mathfrak{P}^{\frac13\mathbb{Z}}$.  Thus, the second term becomes $\bigoplus_{\mathfrak{P}\text{ ramified in }L} \mathfrak{P}^{\frac13\mathbb{Z}/\mathbb{Z}}$.  The question that we are interested in is whether there is an element $u \in \mathcal{O}_L^\times$ such that $N_{L/K}(u) = \zeta_3$, but that is exactly equivalent to $H^2(\langle\sigma\rangle, \mathcal{O}_L^\times) = 0$, which then can be seen to be equivalent to the two statements that $\mathcal{O}_F^\times/(N_{L/F}(F^\times) \cap \mathcal{O}_F^\times) = 0$ and that $\Cl(L)[\sigma - 1]$ is generated by the ramified primes in $L/K$.  At this point, we are prepared to produce the first algebraic criterion for the $\zeta_3$-Pell equation to have a solution.

\begin{theorem}\label{Zeta3PellCondition1}
Write $L = K(\sqrt[3]{\alpha})$ as above, and assume that $\alpha = \pi_1^{a_1} \cdots \pi_n^{a_n}$ with $\pi_i \equiv 1 \pmod{\lambda^2}$ and $a_i = 1$ or $2$.  Then there is a solution to the $\zeta_3$-Pell equation if and only if all the $\pi_i$s satisfy $\pi_i \equiv 1 \pmod{\lambda^3}$ and the primes $\mathfrak{P}_i$ lying above $\pi_i\mathcal{O}_K$ generate $\Cl(L)[\sigma - 1]$.
\end{theorem}
\begin{proof}
Based on the discussion above, what is needed to be shown is that $\alpha$ has the specified form if and only if $\mathcal{O}_F^\times/(N_{L/F}(F^\times) \cap \mathcal{O}_F^\times) = 0$. \\

If $\alpha$ is divisible by a prime $\pi$ with $\pi \not\equiv 1 \pmod{\lambda^3}$ then, writing $K_\pi$ for the completion of $K$ at $\pi$ and $L_\pi$ for $L \otimes_K K_\pi$, one has that $L_\pi$ is a ramified cubic extension of $K_\pi$.  Thus, $N_{L_{\pi}/K_{\pi}}(\mathcal{O}_{L_{\pi}}^\times)$ is the unique index three subgroup of $\mathcal{O}_{K_{\pi}}^\times$.  Since $\pi \not\equiv 1 \pmod{\lambda^3}$, there is no ninth root of unity in $K_{\pi}$, and so $\zeta_3$ isn't a cube in $K_{\pi}$ and consequently not a norm from $\mathcal{O}_{L_{\pi}}^\times$.  Thus, if the first condition is not satisfied, then there is a local obstruction to $\zeta_3$ being a norm of an element of $L^\times$ and so $\mathcal{O}_F^\times/(N_{L/F}(F^\times) \cap \mathcal{O}_F^\times) \neq 0$. \\

As for the other direction, assume that $\alpha$ is of the form described.  Classical calculations (see for example (tbc)) show that $h_{2/1}(\langle\sigma\rangle, \mathcal{O}_L^\times) = 1/3$, and thus one gets that $\dim_{\mathbb{F}_3}(\Cl(L)[\sigma-1]) = n - 1 - \dim_{\mathbb{F}_3}(\mathcal{O}_F^\times/(N_{L/F}(F^\times) \cap \mathcal{O}_F^\times))$ as there are $n$ primes ramified in $L/K$.  But one also has that $\dim_{\mathbb{F}_3}(\Cl(L)[\sigma - 1]) = \dim_{\mathbb{F}_3}(\Cl(L)/(\sigma - 1))$.  Moreover, $L(\sqrt[3]{\pi_1}, \ldots, \sqrt[3]{\pi_{n-1}})$ is a degree $3^{n-1}$ extension of $L$ that is unramified everywhere (the total extension is unramified at $\lambda$ because all of cuberoots are of elements that are $1 \pmod{\lambda}^3$).  Further, it is abelian over $K$ and thus corresponds to a $\sigma-1$ invariant quotient of $\Cl(L)$.  One has then that $n-1 \leq \dim_{\mathbb{F}_3}(\Cl(L)[\sigma - 1]) = n-1-\dim_{\mathbb{F}_3}(\mathcal{O}_F^\times/(N_{L/F}(F^\times) \cap \mathcal{O}_F^\times))$ and so $\mathcal{O}_F^\times/(N_{L/F}(F^\times) \cap \mathcal{O}_F^\times) = 0$, which is what was needed.
\end{proof}
The proof of this theorem also proves the following:
\begin{corollary}\label{GenusField}
With $L$, $\alpha$, and $\pi_i$ as above, and assuming that $\pi_i \equiv 1\pmod{\lambda^3}$ for all $i$, one has that the maximal unramified $\sigma-1$ cotorsion extension of $L$ is $L(\sqrt[3]{\pi_1}, \cdots, \sqrt[3]{\pi_{n-1}})$.
\end{corollary}
\subsection{R\'edei matricies}\label{redei}
Fix $L = K(\sqrt[3]{\alpha})$, and assume that $\alpha = \pi_1^{a_1} \cdots \pi_n^{a_n}$ with $\pi_i \equiv 1\pmod{\lambda^3}$.  The goal now is to undestand the group $\Cl(L)_{\sigma-1}$, which is all of the $(\sigma - 1)$-power torsion in the class group.  We will not come to a complete description, but we will be able to say a bit more than in the previous subsection.\\

The main trick is the following: if $M$ is an $R_{\sigma} = \mathbb{Z}[\sigma]/\langle \sigma^2+\sigma+1 \rangle$-module which is finite as a group, then one may consider the map from $M[\sigma-1] \rightarrow M/(\sigma-1)$.  This is a map from $\mathbb{F}_3$-vector spaces of the same dimension, and so is given by a matrix.  One can compute this rank by appealing to the fact that every $R_{\sigma}$-module that is finite as a group is of the form $\oplus_{\pi_i} \oplus_{a_{ij}} R_\sigma/\pi_i^{a_{ij}}$ where the first sum runs over some primes in $R_\sigma$ and $a_{ij}$s are positive integers.  Assuming that $\pi_1$ is the prime $\sigma-1$, then one has that the dimension of the spaces $M[\sigma-1]$ and $M/(\sigma-1)$ are the number of $a_{1j}$'s, and the rank is the number of $a_{1j}$'s that are equal to one.  Finally, it is actually useful to think of this not as a map but rather as a pairing $M[\sigma - 1] \times M^\vee[\sigma-1] \rightarrow \mathbb{F}_3$, where now the left kernel of this pairing is $(\sigma - 1) M[(\sigma-1)^2]$.  This is obviously equivalent, but in the case of class groups, its simpler to write down the pairing than the actual map.\\

To apply this idea to the setup here, let $\mathfrak{P}_i$ be the prime lying over $\pi_i$.  While it is not clear if the primes $\mathfrak{P}_i$ generate $\Cl(L)[\sigma-1]$, we can still consider the the subgroup of $\Cl(L)[\sigma-1]$ generated by these primes.  There is the relation $\mathfrak{P}_1^{a_1}\cdots\mathfrak{P}_n^{a_n} = \sqrt[3]{\alpha}\mathcal{O}_L$ which says that we only need to consider the subgroup generated by $\mathfrak{P}_i$ for $i$ running from $1$ to $n-1$.  Additionally, we know by Corollary \ref{GenusField} that $\Cl(L)/(\sigma-1) = \Gal(L(\sqrt[3]{\pi_1}, \cdots, \sqrt[3]{\pi_{n-1}})/L)$.  Thus, one has that $\Cl(L)^\vee[\sigma-1]$ is generated by characters $\chi_i$ with $\chi_i(\mathfrak{P}_j) = \left(\frac{\pi_j}{\pi_i}\right)$ for $i \neq j$ and by using the fact that $\chi_i(\mathfrak{P}_1^{a_1} \cdots \mathfrak{P}_n^{a_n}) = 1$ to compute $\chi_i(\mathfrak{P}_i)$.\\

To write down the actual matrix, we will write $\log\left(\left(\frac{\beta}{\pi}\right)_3\right) = a$ for $a \in \mathbb{F}_3$ if $\left(\frac{\beta}{\pi}\right)_3 = \zeta_3^a$.  Then if we let $b_{ii}=a_i^{-1}\log\left(\left(\frac{\pi_i^{a_i}/\alpha}{\pi_i}\right)_3\right)$ and $b_{ij} = \log\left(\left(\frac{\pi_i}{\pi_j}\right)_3\right)$, one has that the matrix $M_L = (b_{ij})$ represents the pairing the subgroup of $\Cl(L)[\sigma-1]$ generated by the primes $\mathfrak{P}_i$ with $\Cl(L)^\vee[\sigma-1]$.
\begin{theorem}\label{PositiveCriterion}
Assume that $M_L$ has full rank.  Then there is a unit $u \in\mathcal{O}_L^\times$ such that $N_{L/K}(u) = \zeta_3$.
\end{theorem}
\begin{proof}
Let $V$ be the free $n-1$-dimensional $\mathbb{F}_3$ vector space with basis vectors $e_i$.  Saying $M_L$ has full rank is the same thing as saying that the composite map $V \rightarrow \Cl(L)[\sigma-1] \rightarrow \Cl(L)/(\sigma-1)$ has full rank (where the first map sends $e_i$ to $\mathfrak{P}_i^{a_i}$ and the second map is the natural one).  But these are all maps between $n-1$-dimensional vector spaces over $\mathbb{F}_3$ and so if the composite map has full rank, so too does the first one.  But that just says that the ramified primes generate $\Cl(L)[\sigma-1]$, which by Theorem \ref{Zeta3PellCondition1} shows that there is a unit $u \in\mathcal{O}_L^\times$ such that $N_{L/K}(u) = \zeta_3$.
\end{proof}

This condition is not necessary for the $\zeta_3$-Pell equation to be solvable.  Indeed, playing around with fields of the form $K(\sqrt[3]{\pi_1\pi_2})$, one sees that a proportion of approximately equal to $2/3$ of such fields satisfy the condition, but an additional proportion roughly equal to $1/36$ do not satisfy the condition but nevertheless have a solution to the $\zeta_3$-Pell equation.

\section{Governing fields}\label{governing}
The goal of this section is to produce a negative criterion for the $\zeta_3$-Pell equation, allowing us to bound from above the number of fields which contains a solution to the equation.

\begin{theorem}\label{NegativeCondition}
Keeping notation as above, if one has that $\mathrm{dim}(\mathrm{ker}(M_L)) > \mathrm{dim}((\sigma-1)\Cl(L)/(\sigma-1)^2)$, then the $\zeta_3$-Pell equation is insoluble.
\end{theorem}
\begin{proof}
As discussed above, if there is a solution to the $\zeta_3$-Pell equation, one has that $M_L$ is the actual R\`edei pairing and not just the pairing restricted to a subgroup of the $(\sigma - 1)$-torsion.  Thus, one would have $\mathrm{dim}(\mathrm{ker}(M_L)) = \mathrm{dim}((\sigma-1)\Cl(L)/(\sigma-1)^2)$ in this case, contradicting the assumption in the Theorem.
\end{proof}

Before constructing the governing fields in full generality, we find it illuminating to focus on one simple case.  We will look at fields of the form $L = K(\sqrt[3]{17\pi})$.  Observe that there is a diagram of fields with the field corresponding to $\Cl(L)/(\sigma - 1)$ on top:
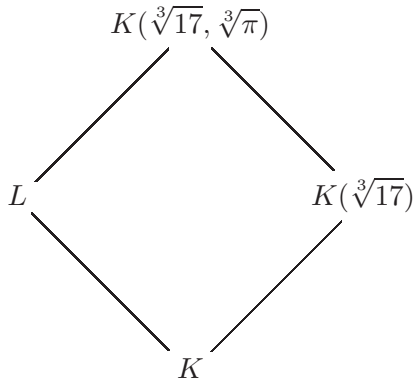
\begin{diagram}
 & & K(\sqrt[3]{17}, \sqrt[3]{\pi}) \\
 & \ldLine & & \rdLine \\
 L & & & & K(\sqrt[3]{17}) \\
 & \rdLine & & \ldLine \\
 & & K
\end{diagram}
As discussed before, the story is over if $\left(\frac{\pi}{17}\right)_3 \neq 1$: under this assumption we have that all of the $(\sigma-1)$-power part of the class group is $(\sigma-1)$-torsion.  Thus, we will assume that $\left(\frac{\pi}{17}\right)_3 \neq 1$.  We want to know if there is a an unramified extension $L_2/L$ such that $\Gal(L_2/L) = R_{\sigma}/(\sigma-1)^2$.  Writing $F = K(\sqrt[3]{17})$ and $\tau$ for the generator of $\Gal(F/K)$, we have any such field is also abelian over $F$ with Galois group $R_\tau/(\tau - 1)^2$, ramified only at $\pi$.  But now class field theory intervenes and we can compute the maximal extension of $F$, unramified away from $\pi$, with Galois group killed by $\tau - 1$.  Denoting this field by $F_\pi$, we have that $\Gal(F_\pi/F)$ is surjected onto by $V_\pi := (\mathcal{O}_F/\pi)^\times \otimes_\mathbb{Z} \mathbb{F}_3$, and the kernel of this surjection is generated by $V_\pi[\tau - 1]$ and the image of $\mathcal{O}_L^\times$.  Thus, there is a field $L_2/F$ if and only if $\mathcal{O}_L^\times \subset V_{\pi}[\tau-1]$, which happens if and only if $(\mathcal{O}_F^\times)^{\tau-1}$ all reduce to cubes mod $\pi$.  But that is tantamount to saying that $\pi$ splits in $F(\sqrt[3]{(\mathcal{O}_F^\times)^{\tau-1}})$.  One can easily compute that $\zeta_9$ is in this field, and that the degree of the extension $F(\sqrt[3]{(\mathcal{O}_F^\times)^{\tau-1}})/F(\zeta_9)$ is 3.  Thus, one gets that, among all primes $\pi \equiv 1 \pmod{\lambda^3}$, two-thirds of them don't have $\left(\frac{\pi}{17}\right)_3 = 1$ and thus have a solution to the $\zeta_3$-Pell equation, two-thirds of the remainder do have $\left(\frac{\pi}{17}\right)_3 = 1$ but don't split in $F(\sqrt[3]{(\mathcal{O}_F^\times)^{\tau-1}})$ and so don't have a solution to the $\zeta_3$-Pell equation.  The authors as of right now have no ideas for how to get deeper into this set.\\

To set up the full generality, we will consider fields of the form $L_\pi = K(\sqrt[3]{\alpha\pi})$ with $\alpha = \pi_1^{a_1} \cdots \pi_n^{a_n}$.  Let $F = K(\sqrt[3]{\pi_1}, \ldots, \sqrt[3]{\pi_n})$, and write $\Gal(F/K) = \langle \tau_1, \ldots \tau_n\rangle$ with $\tau_i(\sqrt[3]{\pi_i}) = \zeta_3 \sqrt[3]{\pi_i}$ and $\tau_i(\sqrt[3]{\pi_j}) = \sqrt[3]{\pi_j}$ for $i \ne j$.  Choose a matrix $M = (b_{ij})$ such that $b_{ij} = \log\left(\left(\frac{\pi_i}{\pi_j}\right)_3\right)$ for $i \neq j$.  Writing $M_\pi$ for the matrix associated to the pairing as constructed in subsection \ref{redei}, one has that there is an element $\tau$ in $\Gal(F/K)$ such that $\Frob_{\pi} = \tau$ if and only if $M_\pi = M$.  The main theorem is as follows:
\begin{theorem} \label{gov field} 
There is a field $F_\alpha/F(\zeta_9)$ of degree $\dim(\ker(M))$ such that $F_\alpha$ is Galois over $K$, and, for all $\pi$ with $M_\pi = M$, one (equivalently any) prime $\mathfrak{P}$ lying over $\pi$ in $F$ splits in $F_\alpha$ if and only if $\dim((\sigma - 1)\Cl(L_\pi)[(\sigma-1)^2]) = \dim(\ker(M))$.
\end{theorem}
\begin{proof}
What follows will be essentially the same argument, but with more notation necessary.  Let $r_1: \mathbb{F}_3\langle e_1, \ldots, e_n \rangle \rightarrow \Cl(L)/(\sigma-1)$ sending $e_i$ to the class associated to the prime over $\pi_i$.  The assumptions on $\pi$ imply that the cokernel of $r_1$ is independent of $i$, and so we may view $W := \coker(r_1)$ as a fixed subgroup of $\Gal(F/K)$.  Let $R = \mathbb{F}_3 + W$, a ring where one defines multiplication by setting $\tau \tau' = 0$ for any two $\tau, \tau' \in W$.  $R$ is naturally a quotient of the group ring $\mathbb{F}_3[\Gal(F/K)]$.\\

Define $V_\pi := (\mathcal{O}_F/\pi)^\times \otimes_\mathbb{Z} \mathbb{F}_3$ and $V_\pi' = V_\pi \otimes_{\mathbb{F}_3[\Gal(F/K)]} R$.  The assumptions on $\pi$ imply that $V_\pi'$ is isomorphic to $R$ as an $R$-module.  Write $L_2$ to be the maximal abelian extension of $L$ unramified everywhere and whose Galois group is $(\sigma-1)^2$-torsion.  One has then that $L_2/F$ is abelian and unramified outside of $\pi$, and $\Gal(L_2/F)$ is an $R$-module.  Class field theory gives a map from $V_\pi' \rightarrow \Gal(L_2/F)$ which is surjective, and whose kernel is the image of global units.  However, one has that $\dim_{\mathbb{F}_3}(\Gal(L_2/F)) = \dim((\sigma - 1)\Cl(L_\pi)[(\sigma-1)^2]) + 1$, so one has $\dim((\sigma - 1)\Cl(L_\pi)[(\sigma-1)^2]) = \dim(\ker(M))$ if and only if there is no kernel in the map from $V_\pi' \rightarrow \Gal(L_2/F)$.  Asking that the image of $\mathcal{O}_F^\times$ in $V_\pi'$ is trivial is just asking that $r \cdot u$ is a cube for all $r \in \ker(\mathbb{F}_3[\Gal(F/K)] \rightarrow R)$ and $u \in \mathcal{O}_F^\times$.  This gives the existence of the field $F_\alpha/F$ by taking the corresponding Kummer extension; all that's left is to compute the degree.\\

Now, the structure of $\mathcal{O}_F^\times \otimes_\mathbb{Z} \mathbb{F}_3 = \mathfrak{m}$ as an $\mathbb{F}_3[\Gal(F/K)]$-module where $\mathfrak{m}$ is the unique maximal ideal in $\mathbb{F}_3[\Gal(F/K)]$.  Asking that $(\mathcal{O}_F^\times \otimes_\mathbb{Z} \mathbb{F}_3) \otimes_{\mathbb{F}_3[\Gal(F/K)]} R$ maps to $0$ in $V_\pi'$ is imposing $\dim(W)$ different conditions, as $\mathfrak{m} \otimes_{\mathbb{F}_3[\Gal(F/K)]} R = W$.  That gives the degree of $F_\alpha$ and completes the proof of the Theorem.
\end{proof}

At this point, one can imagine how the proof of the main theorem goes.  Firstly, one wants to show that the matricies $M_L$ look like large random symmetric matricies over $\mathbb{F}_3$.  This gives the lower bound on the number of fields that solve the $\zeta_3$-Pell equation.  This also gives a distribution on the dimension of the kernels of the matricies $M_L$, and the next step is to show that for fields $L$ with $\dim(\ker(M_L)) = d$, there is a $\frac{3^d - 1}{3^d}$ chance that Theorem \ref{NegativeCondition} applies.  This gives the upper bound.

\section{Distribution of the $3$-rank}\label{Analysis}
This Section consists of two parts.  First, we compute the ranks of the matrices $M_L$ for varying $L$, and then we use that to compute the dimensions of $(\sigma - 1)(\Cl(L)[(\sigma-1)^2])$ for varying $L$.

\subsection{Ranks of Matricies}
For an element $s \in \bZ[\zeta_3]$, we denote by $N(s)$ its norm. We will require the following Lemma regarding cubic character sums: 

\begin{lemma} \label{double osc} Let $\{\alpha_m\}, \{\gamma_n\}$ be two complex sequences indexed by the Eisenstein integers such that $|\alpha_m|, |\gamma_n| \leq 1$ for all $m, n \in \bZ[\zeta_3]$. Let $\mu$ be the natural extension of the M\"{o}bius function to $\bZ[\zeta_3]$. Put
\[\Xi(M,N, \alpha, \beta) = \sum_{N(m) \leq M} \sum_{N(n) \leq N} \alpha_m \gamma_n \mu^2(m) \mu^2(n) \left(\frac{m}{n}\right)_3, \]
and 
\[S_1(M,N) = N^{-1/2} + M^{-1/4} N^{1/2}, S_2(M,N, \ep) = M^\ep \left(N^{-1/8} + M^{-1/4} N^{1/8} \right).\]
Then for all $\ep > 0$
\begin{equation} \Xi(M,N, \alpha, \beta) = O_\ep \left(MN \min\left\{S_1(M,N), S_1(N,M), S_2(M,N,\ep), S_2(N,M, \ep) \right\}\right).
\end{equation} 

\end{lemma}

\begin{proof} See Proposition 9 in \cite{FK2} and Remark 6.3, as well as \cite{HB}. 
\end{proof}

We put 
\begin{equation} S(X) = \{L = K(\sqrt[3]{\alpha}) : |\Delta_L| \leq X, \pi | \alpha \Rightarrow N(\pi) \equiv 1 \pmod{9}\}
\end{equation}
and for $n \geq 1, n \in \bZ$,
\begin{equation} S_n(X) = \{L \in S(X) : \dim \ker M_L = n\}.
\end{equation}
We wish to show that the limit
\[\lim_{X \rightarrow \infty} \frac{|S_n(X)|}{|S(X)|}\]
exists and is equal to the corresponding limit for random $\bF_3$-matrices.  Precisely, one defines $\beta_n := \frac{\beta}{\prod_{j = 1}^n (3^j - 1)}$, and we wish to show
\begin{theorem}
One has that 
$$\lim_{X \rightarrow \infty} \frac{|S_n(X)|}{|S(X)|} = \beta_n.$$
\end{theorem}
That $\beta_n$ gives the right distribution is classical, and can be seen directly from Proposition 9 in \cite{SteRedei}.
To wit, we decompose $S_n(X)$ as
\begin{align*} S_n(X) & = \bigcup_{r \geq 0} \{L : |\Delta_L| \leq X, \dim \ker M_L = n, \omega(\alpha) = r+1\} \\
& = \bigcup_{r \geq 0} S_n^{(r)}(X) \\
& = \bigcup_{r \geq 0} \bigcup_{\substack{M \text{ symmetric } r \times r \text{ } \bF_3 \text{ matrix} \\ \rk M = n - r}} \{L : |\Delta_L| \leq X, M_L = M, \omega(\alpha) = r + 1\}.
\end{align*}
We denote by $S_n^{(r)}(M; X)$ a term in the last line. \\

We wish to write down the condition $M_L = M$ in terms of cubic symbols. For a fixed $a_{ij} \in \bF_3$, we have that the indicator function for $a_{ij} = m_{ij}$ is given by
\begin{equation} \label{uij} u_{ij} = \frac{1 + \zeta_3^{-a_{ij}} \left(\frac{\pi_i}{\pi_j}\right)_3 + \zeta_3^{-2a_{ij}} \left(\frac{\pi_i}{\pi_j}\right)_3^2 }{3}.\end{equation}

We may assume going forward that the number of prime factors is bounded by
\begin{equation} \label{r bd} \log \log X - (\log \log X)^{3/4} < r < \log \log X + (\log \log X)^{3/4}.\end{equation} 
This follows from an appropriate application of the Selberg-Delange method, and a corresponding refinement to Theorem 5 in \cite{Ten}. From here, we shall estimate $S_n^{(r)}(A; X)$ individually, being careful to produce error terms that depend explicitly on $r$ and otherwise uniform. \\

We put
\[A = \left(a_{ij} \right), a_{ij} \in \bF_3, a_{ij} = a_{ji}, 1 \leq i, j \leq r \]
for a fixed $r \times r$ symmetric $\bF_3$-matrix. Recall from (\ref{uij}) that we can write
\[\Ind(M_L = A) = \prod_{i \leq j} u_{ij}. \]
Expanding, we see that we are left to deal with sums of the shape
\begin{equation} \label{osc sum} \sum_{N(\pi_1 \cdots \pi_r) \leq X} 3^{-\frac{r^2 + r}{2}} \prod_{i \leq j} \left(\frac{\pi_i}{\pi_j} \right)_3^{b_{ij}},
\end{equation}
where $b_{ij} \in \{0,1,2\}$. Observe that the term with $b_{ij} = 0$ for all $i,j$ is expected to be the main term. \\

Put
\begin{equation} T_r(D;X) = \{\alpha = \pi_1 \cdots \pi_r \leq X : N(\pi_i) > D, N(\pi_i) \equiv 1 \pmod{9}  \text{ for } 1 \leq i \leq r\}.
\end{equation}
In other words, $T_r(D;N)$ is the set of elements in $\O_K$ having exactly $r$ prime factors (up to multiplicity) and norm bounded by $X$, such that the norm of each prime factor exceeds $D$ and is congruent to $1 \pmod{9}$. \\

For an element $n = p_1 \cdots p_r \in T_r(D;X)$, put
\[d_{2,r} = d_{2,r}(n) = \prod_{2 \leq i \leq r} N(\pi_i).\]
Let $\T_r(D;X)$ denote the subset of $T_r(D;X)$ having the property that
\[d_{2,r} > X e^{-\exp\left(\sqrt{\log \log X}\right)}.\]

We require:

\begin{lemma} \label{TT comp} Let $\sideset{}{_r^\dagger} \sum$ denote summating over the range 
\[\log \log X - (\log \log X)^{3/4} < r < \log \log X + (\log \log X)^{3/4}.\]
Then 
\[\lim_{N \rightarrow \infty} \frac{\sideset{}{_r^\dagger} \sum \left \lvert  \T_r(D;X) \right \rvert}{\sideset{}{_r^\dagger} \sum |T_r(D;X)|} = 0.\]
\end{lemma} 

\begin{proof} By the arguments given in Section 5 of \cite{Smith2}, we have the estimate
\[k_1 \frac{rX(\log \log X)^{r-1}}{6^r \log X} < |T_r(D;X)| < k_2 \frac{rX (\log \log X)^{r-1}}{6^r \log X}\]
for some positive numbers $k_1, k_2$ and sufficiently large $X$. Indeed, we can even take $D$ to be a slowly increasing function of $N$; taking $D = \max\{1, \log \log \log X\}$ suffices. \\

Using the bound 
\[|\T_r(D;X)| \leq \sum_{\substack{p_1 = 1 \\ p_1 \equiv 1 \pmod{9}}}^{\exp(\exp(\sqrt{\log \log X}))} T_{r-1}(D; X),\] 
we then see that
\begin{equation} |\T_r(D;X)| \leq \sum_{p_1} \frac{k_2 r(X/p_1) (\log \log X)^{r-2}}{6^r \log X} \ll \frac{rX(\log \log X)^{r-2}}{6^r \log X} \sum_{p_1}^{\exp(\exp(\sqrt{\log \log X}))} \frac{1}{p_1}.
\end{equation}
By Dirchlet's theorem applied to primes congruent to $1$ mod $9$, we find that
\[\sum_{\substack{p_1 = 1 \\ p_1 \equiv \pmod{9}}}^{\exp(\exp(\sqrt{\log \log X}))} \frac{1}{p_1} = \frac{(\log \log X)^{1/2}}{6} + B_0 + O\left((\log X)^{-1}\right),
\]
where $B_0$ is an absolute constant. It therefore follows that
\begin{equation} |\T_r(D;X)| = O\left(r 6^{-r} X (\log \log X)^{r-3/2} (\log X)^{-1}\right),
\end{equation}
and thus we see that 
\[\frac{|\T_r(D;X)|}{|T_r(D;X)|} = O \left((\log \log X)^{-1/2}\right),\]
which tends to $0$ as $X \rightarrow \infty$, as desired. 
\end{proof}
We may thus restrict our attention to the set:

\begin{equation} T_r^\ast(D;X) = \left \{m \in T_r(D;X) : p | m \Rightarrow p > \exp \left(\exp\left(\sqrt{\log \log X}\right)\right) \right\}.
\end{equation} 

Put
\begin{equation} \label{X dag} X^\dagger = \exp \left(\exp \left(\sqrt{\log \log X} \right) \right),
\end{equation}
and consider the contribution from those sums (\ref{osc sum}) such that there exist two indices $u, v$ for which $b_{uv} \in \{1,2\}$ and such that $B_u, B_v \gg X^\dagger$. Denote this subset by $S^\dagger(\BB; X)$. Put
\begin{equation} \Xi(u,v) = \prod_{w \ne u,v} \left(\frac{\pi_u}{\pi_w} \right)_3^{b_{uw}}.
\end{equation}
We then have the bound
\begin{equation} \left \lvert S^\dagger(\BB; X) \right \rvert \leq \sum_{\pi_w : w \ne u,v} \prod_{w \ne u,v} 3^{-(r^2 + r)/2} \left \lvert \sum_{\pi_u} \sum_{\pi_v} \Xi(u,v) \Xi(v,u) \left(\frac{\pi_u}{\pi_v} \right)_3^{b_{uv}}\right \rvert,
\end{equation}
and $b_{uv} \in \{1,2\}$. The upshot is that $ |\Xi(u,v)| \leq 1$ so Lemma \ref{double osc} applies. We thus obtain the bound, for all $\ep > 0$,
\[\left \lvert S^\dagger(\BB; X) \right \rvert \ll_\ep \left(\prod_{w \ne u,v} B_w \right) \left(B_u B_v \left(B_u^{-1/8 + \ep} + B_v^{-1/8 + \ep} \right) \right) \ll_\ep X \left(X^\dagger \right)^{-1/8 + \ep}, \]
which gives the bound
\begin{equation} \label{X dag bd} \left \lvert S^\dagger(\BB; X) \right \rvert \ll X \exp \left(-c \exp\left( \sqrt{\log \log X}\right) \right)
\end{equation}
for some $c > 0$. This is smaller than $X (\log X)^{-A}$ for any $A > 0$, which is enough for us. \\

We have thus shown that all sums of the form (\ref{osc sum}) with at least one pair of indices $\{i,j\}$ with $b_{ij} \ne 0$ contributes a negligible amount to $S_n^{(r)}(A; X)$, and hence only the main term where $b_{ij} = 0$ for all $i, j$ contributes.  

\subsection{Dimensions of Parts of Class Groups}
Our goal is to translate Theorem \ref{gov field} into a statement about the distribution of $3$-ranks of $L = K(\sqrt[3]{\alpha})$ over $K$, as $\alpha$ ranges over $\O_K$ with bounded norm. \\

We first extract a probabilistic consequence of Theorem \ref{gov field}:

\begin{proposition} \label{1 prime} Let $\alpha \in \O_K$ with $\alpha = \pi_1 \cdots \pi_r$. Put $L_\pi = K(\sqrt[3]{\alpha \pi})$. Fix a symmetric $r \times r$ $\bF_3$-matrix $M = (m_{ij})$, and suppose that 
\[m_{ij} = \log \left(\frac{\pi_i}{\pi_j}\right)_3 \text{ for } i \ne j.\]
Fix a coset $c \in \left(\O_K/(\alpha)\right)^\ast$ such that $\pi \in c$ if and only if $M_{L_\pi} = M$. Let $m = \dim \ker(M)$. Let $\L_c(X)$ be the set of primes $\pi \in \O_K$ with $N(\pi) \leq X$ and $N(\pi) \equiv 1 \pmod{9}$, with $\pi \in c$. Then
\[\lim_{X \rightarrow \infty} \frac{\left \lvert \pi \in \L_c(X) : \dim (\sigma - 1) \Cl(L_\pi)[(\sigma - 1)^2] = m \right \rvert}{\left \lvert \L_c(X) \right \rvert} =  3^{-m}. \]
Likewise, 
\[\lim_{X \rightarrow \infty} \frac{\left \lvert \pi \in \L_c(X) : \dim (\sigma - 1) \Cl(L_\pi)[(\sigma - 1)^2] \ne m \right \rvert}{\left \lvert \L_c(X) \right \rvert} =  1 - 3^{-m}. \]
\end{proposition}

\begin{proof} This follows from Theorem \ref{gov field} and Chebotarev's density theorem applied to the field $F(\alpha)$. 
\end{proof}

In order to obtain the analogous result where we allow more variation than one prime at a time, we follow the strategy in \cite{Smith1}, augmented by the observation in \cite{Smith2} that one can obtain a quantitative version of Chebotarev's density theorem whose dependency on the conductor $q$ of a number field is a power of the logarithm. The following result is a slight generalization of Proposition 6.5 in \cite{Smith2}.

\begin{proposition} \label{Chebo} Let $M/\bQ$ be a finite Galois extension and $G = \Gal(M/\bQ)$ be a $\ell$-group for a prime $\ell$. Suppose $M = K L$, where $L/\bQ$ is is a Galois extension of degree $d$ and $K/\bQ$ is an elementary abelian extension, and where the discriminant $\Delta(L)$ of $L/\bQ$ and the discriminant $\Delta(K)$ of $K/\bQ$ are relatively prime. Let $d(K_0)$ be the maximal discriminant of degree $p$ subfield $K_0 \subset K$. \\

Let $f : G \rightarrow [-1,1]$ be a class function of $G$ with average over $G$ equal $\kappa$. Then there is a number $c_\ell$ depending at most on $\ell$ such that
\[\frac{1}{|G|} \sum_{p \leq X} f \left(\left[\frac{M/\bQ}{p} \right] \right) \log p = \kappa X +  \]
\[ O_\ell \left( X^\gamma + X \exp \left( \frac{c_\ell d^{-4} \log X}{\sqrt{\log X} + 3d \log |\Delta(K) \Delta(L)} \right) \left(d^2 \log \left \lvert X \Delta(K) \Delta(L) \right \rvert \right)^4  \right) \]
for all $X \geq 3$, where $\gamma$ is the maximal real zero of any Artin $L$-function defined for $G$, the term being ignored if no such zero exists. The implied constant depends at most on $\ell$.
  \end{proposition}
  
  \begin{proof} Let $\rho$ be a non-trivial irreducible representation of $G$. As $G$ is a $p$-group, $G$ is nilpotent and is hence a monomial group. Hence the Artin conjecture is true for $\rho$, which means that the Artin $L$-function $L(\rho,s)$ is entire. The representations $\rho \otimes \rho$ and $\rho \otimes \ol{\rho}$ also satisfy Artin's conjecture, so $L(\rho \otimes \ol{\rho}, s)$ is entire except for a simple pole at $s = 1$ and $L(\rho \otimes \rho, s)$ is entire unless $\rho$ is isomorphic to $\ol{\rho}$. \\
  
  We now apply Theorem 5.10 in \cite{IK} to $L(\rho, s)$. Note that $\rho$ is defined on $\Gal(K_0 L/\bQ)$ for some degree-$\ell$ extension $K_0/\bQ$ inside $K$, so the degree of $L(\rho,s)$ is bounded by $2d$ and the conductor of $L(\rho,s)$ is bounded by the discriminant of $K_0 L/\bQ$. This in turn is bounded by $d(K_0)^d \cdot |\Delta(L)|^\ell$. Theorem 5.13 in \cite{IK} then gives
  \[\sum_{p \leq X} \chi_\rho \left(\left[\frac{M/\bQ}{p} \right] \right) \log p \]
  \[ = O_\ell \left(X^\gamma + X \exp \left(\frac{-c_\ell d^{-4} \log X}{\sqrt{\log X} + 3d \log |d(K_0) \Delta(L)|}\right) (d^2 \log |X d(K_0) \Delta(L)|)^4 \right).\]
  Now we may write $F$ in the form $\sum_{\rho} a_\rho \chi_\rho$, the sum running over irreducible representations of $G$. Then
  \begin{align*} \sum_\rho |a_\rho| & = \sum_\rho \left \lvert |G|^{-1} \sum_{g \in G} f(g) \cdot \ol{\chi_\rho}(g) \right \rvert \\
  & \leq \sum_\rho \left(\frac{1}{|G|} \sum_{g \in G} f(g) \cdot \ol{f}(g) \right)^{1/2} \left(\frac{1}{|G|} \sum_{g \in G} \chi_\rho(g) \ol{\chi_\rho}(g) \right)^{1/2} (\text{by Cauchy-Schwarz}) \\
  & \leq \sum_\rho 1 \\
  & \leq |G|. 
  \end{align*}
  The Proposition follows from linearity. 
  \end{proof}
  
In order to apply Proposition \ref{Chebo}, we will need to bound the conductor of $F(\alpha)$. This follows from the fact that $F(\alpha)$ is unramified over $K(\sqrt[3]{\alpha})$ and degree $3^{r + m}$ over $K(\sqrt[3]{\alpha})$. In particular, the conductor of $F(\alpha)$ is at most
\[N(\alpha)^{3^{r+m}}.\]
The implied constant in the error term in Proposition \ref{Chebo} depends at most on the prime $\ell$, which we may view as an absolute constant (indeed, for our application $\ell = 3$). We then see that 
\[ \exp \left(\frac{-c_\ell d^{-4} \log X}{\sqrt{\log X} + 3d \log |d(K_0) \Delta(L)|} \right) (d^2 \log |X d(K_0) \Delta(L)|)^4  \ll_A  (\log X)^{-A} \]
for any $A > 0$. Thus it remains to treat the possible Siegel zero in the error term. Here we merely observe that the number fields $F(\alpha)$ under consideration are $3$-extensions of $K = \bQ(\zeta_3)$, hence the only degree-1 factors which appear in the decomposition for the Artin $L$-function of $F(\alpha)$ correspond to the trivial representation or the field character of $K$, namely $\chi_{-3}$. Therefore the Artin $L$-function of $F(\alpha)$ cannot have Siegel zeroes and we can ignore this contribution. \\

Applying Proposition \ref{Chebo} to Proposition \ref{1 prime} then gives
\begin{equation} \label{PS1} \sum_{\substack{\pi \in \L_c(X^\dagger)}} 1 = \frac{\Li(X^\dagger)}{3^{r+1}} + O_a \left((r+m)^4 X^\dagger \exp \left(-(\log X^\dagger)^{a} \right) \right)
\end{equation}
and
\begin{equation} \label{PS2} \sum_{\substack{\pi \in \L_c(X^\dagger) \\ \dim(\sigma - 1) \Cl(L_\pi)[(\sigma - 1)^2] = m}} 1 = \frac{\Li(X^\dagger)}{3^{m + r+1}} + O_a \left((r+m)^4  X^\dagger \exp \left( - \log( X^\dagger)^{a} \right) \right)
\end{equation}
for some $0 < a < 1$. Both follows from Proposition \ref{Chebo}. Note that the error terms in (\ref{PS1}) and (\ref{PS2}) are independent of $\alpha$, whence we may sum both uniformly across $\alpha$ to obtain the following:

\begin{theorem} The limit
\[\lim_{X \rightarrow \infty} \frac{\left \lvert \{\alpha \in S_m(X) : \dim (\sigma - 1) \Cl(K(\sqrt[3]{\alpha})[(\sigma - 1)^2] = m  \} \right \rvert}{|S_m(X)|} = 3^{-m}\]
holds. 
\end{theorem}
We will find it useful to name the set $\{\alpha \in S_m(X) : \dim (\sigma - 1) \Cl(K(\sqrt[3]{\alpha})[(\sigma - 1)^2] = m  \} := S'_m(X)$.

\section{Proof of Theorem \ref{MainTheorem} and Heuristics Behind Conjecture \ref{FalseConjecture}}\label{Proof}
Now, we put together all of the previous discussion into a proof of Theorem \ref{MainTheorem}.
\begin{proof}[Proof of Theorem \ref{MainTheorem}]
The upshot of the disussion in Section \ref{Analysis} is that 
$$\lim_{X\rightarrow\infty} \frac{|S_m(X)|}{|S(X)|} = \beta_m,$$ and
$$\lim_{X\rightarrow\infty} \frac{|S'_m(X)|}{|S(X)|} = \frac{\beta_m}{3^m}.$$
If $\alpha \in S_0(X)$, then the $\zeta_3$-Pell equation for $K(\sqrt[3]{\alpha})$ has a solution by Theorem \ref{PositiveCriterion}, and if $K(\sqrt[3]{\alpha})$ has a solution to the $\zeta_3$-Pell equation, then $\alpha \in S'_m(X)$ for some $m$ by Theorem \ref{NegativeCondition}.  One clearly has $\beta_0 = \beta$, which gives the lower bound in Theorem \ref{MainTheorem}.  To get the upper bound, we need to compute the sum $\sum_{m = 0}^{\infty} \frac{\beta_m}{3^m}.$  Following \cite{FK2}, we expand the sum as follows:
\begin{align*}
\sum_{m = 0}^{\infty} \frac{\beta_m}{3^m} & = \beta\sum_{m=0}^\infty\frac{3^{-m}}{(3-1)(3^2-1)\cdots(3^m-1)} \\
& = \beta\sum_{m=0}^\infty\frac{3^{-m}3^{-\frac{m(m+1)}{2}}}{(1-1/3)(1-(1/3)^2) \cdots (1-(1/3)^m)} \\
& = \beta\prod_{m=0}^{\infty} (1+ (1/3)^{m+1}) \\
& = \prod_{m=0}^\infty \frac{(1+ (1/3)^{m+1})}{(1+ (1/3)^{m})} \\
& = \frac{3}{4}.
\end{align*}
The third equality is just Lemma 4 of \cite{FK2}.  This gives the upper bound in Theorem \ref{MainTheorem}.
\end{proof}

Additionally, here are the considerations that lead one to make Conjecture \ref{FalseConjecture}.  The analogous prediction to Stevenhagens in \cite{Ste} is that, for $\alpha \in S_m(X)$, there is a $\frac{2}{3^{m+1} - 1}$ chance that there is a solution to the $\zeta_3$-Pell equation for $K(\sqrt[3]{\alpha})$.  Thus, one computes
\begin{align*}
\sum_{m=0}^{\infty} \frac{2\beta}{(3^{m+1} - 1)\prod_{j = 0}^{m} (3^j - 1)} & = 2 \sum_{m=1}^{\infty} \frac{\beta}{\prod_{j = 0}^{m} (3^j - 1)} \\
& = 2 \left(\left(\sum_{m=0}^\infty \frac{\beta}{\prod_{j = 0}^{m} (3^j - 1)}\right) - \beta\right) \\
& = 2 \left(\left(\sum_{m=0}^\infty \beta_m\right) - \beta\right) \\
& = 2(1-\beta).
\end{align*}

\section{Differences between the $\zeta_3$-Pell equation and the negative Pell equation}
\label{diffs} 

The goal of this section is two-fold.  Firstly, we wish to convince the reader that the obvious generalization of Stevenhagen's conjecture is likely incorrect.  Secondly, we wish to explain that things behave differently when we allow $\lambda$ to ramify. 
\subsection{The validity of conjecture \ref{FalseConjecture}}

To the first end, we will discuss some of Stevenhagen's considerations when making his conjecture, and show that these do not have exact analogues in the $\zeta_3$-Pell case.  Consider fields $L$ of the form $\mathbb{Q}(\sqrt{pq})$ with $p, q \equiv 1 \pmod{4}$.  Stevenhagen predicts that for half of these we have $\Cl(L)_2 = \mathbb{Z}/2\mathbb{Z}$ and there is a solution to the negative Pell equation. Half of the remaining half would have $\Cl(L)_2 = \mathbb{Z}/2\mathbb{Z}$ but there is no solution to the negative Pell equation. Half of those would have $\Cl(L)_2 = \mathbb{Z}/4\mathbb{Z}$ and there is a solution to the negative Pell equation, and so on. \\

Moving over to the case where $L = K(\sqrt[3]{\pi_1\pi_2})$, the story starts off the same: two-thirds of the time, $\Cl(L)_{\sigma-1} = R_{\sigma}/(\sigma - 1)$ and there is a solution to the $\zeta_3$-Pell equation, and two-thirds of the remaining time $\Cl(L)_{\sigma-1} = R_{\sigma}/(\sigma - 1)$ but there is no solution to the $\zeta_3$-Pell equation.  However, we have the following fact that causes everything to go off the rails:
\begin{proposition}
Assume that $\left(\frac{\pi_1}{\pi_2}\right)_3 = 1$ and there is a solution to the $\zeta_3$-Pell equation in $L$.  Then $\Cl(L)[(\sigma - 1)^3] = R_\sigma/(\sigma-1)^3$.
\end{proposition}
This proposition implies that the case ``$\Cl(L)_{\sigma - 1} = R_{\sigma}/(\sigma-1)^2$, $\zeta_3$-Pell has a solution'' is skipped.  Numerical calculations seem to suggest that this is the only case that is skipped.
\begin{proof}
Since there is a solution to the $\zeta_3$-Pell equation, one has that $\Cl(L)[\sigma-1]$ is generated by $\mathfrak{P}_1$, the prime lying over $\pi_1$.  Letting $L_1 = K(\sqrt[3]{\pi_1}, \sqrt[3]{\pi_2})$ so that $\Gal(L_1/L) = Cl(L)/\sigma-1$, one sees that the first R\`edei map is zero.  Thus, there is a field $L_2/L$ as in the middle of Section \ref{governing} replacing $17$ with $\pi_1$ and $\pi$ with $\pi_2$.  Letting $F = K(\sqrt[3]{\pi_1})$, $\tau$ generates $\Gal(F/K)$, and $V = (\mathcal{O}_F/\pi_2)^\times \otimes_\mathbb{Z} \mathbb{F}_3$, so we have that $\Gal(L_2/F) = V_\pi/V_\pi[\tau-1]$.  Now we need to compute the image of the Frobenius of the prime lying over $\pi_1$ in $\Gal(L_2/L)$.  But this is zero if and only if that's the case for $L_2/F$. We know what that is: we can write $V = ((\mathcal{O}_K/\pi_2)^\times \otimes_\mathbb{Z} \mathbb{F}_3)^3$ with $\tau$ permuting the factors.  Choosing an element $a$ such that $a^3 \equiv \pi_1 \pmod{\pi_2}$, we get that the Frobenius is just the image of $(a, \zeta_3 a, \zeta_3^2 a)$ in $\Gal(L_2/F)$.  By our long-running assumption on primes in $K$, we have that $\zeta_3$ is a cube, so this lies in the diagonal of $V$.  But $V[\tau-1]$ is exactly equal to the diagonal, so the image of the Frobenius of the prime lying over $\pi_1$ in $\Gal(L_2/L)$ is 0, which gives the proposition. \end{proof}

\subsection{Ramification at $\lambda$}

Let $$S_{\lambda}(X) = \{\alpha | N_{K/\mathbb{Q}}(D_{K(\sqrt[3]{\alpha})/K}) < X,  \pi \equiv 1 \pmod{\lambda^3} \forall \pi|\alpha, \pi\neq\lambda \}$$
and 
$$ S^{\zeta_3}_{\lambda}(X) = \{ \alpha \in S_{\lambda}(X) : \exists u \in K(\sqrt[3]{\alpha}) \text{ with } N_{K(\sqrt[3]{\alpha})/K}(u) = \zeta_3\}.$$ 
First off, one can get trivial bounds by observing that $\displaystyle{\lim_{X\rightarrow\infty}} \frac{|S(X)|}{|S_{\lambda}(X)|} = \frac{27}{31}$, so one can instantly get $\frac{27\beta}{31}$ as a lower bound and $\frac{97}{124}$ as an upper bound.  However, it is possible to obtain improvements on this.\\

One can split $S_{\lambda}(X)$ into three sets: $S(X)$, $S'(X)$, and $S''(X)$, where 
\[S^\prime(X) = \{\alpha\in S_{\lambda}(X)|v_{\lambda}(D_{K(\sqrt[3]{\alpha})/K}) = 3\}\] 
and 
\[S^{\prime \prime}(X) = \{\alpha\in S_{\lambda}(X)|v_{\lambda}(D_{K(\sqrt[3]{\alpha})/K}) = 4\}.\]  As mentioned, $27/31$ of the elements in $S_{\lambda}(X)$ lie in $S(X)$ and that $2/31$ of all elements lie in $S^\prime(X)$ and $S^{\prime \prime}(X)$. \\

Now, everything that happened for $S(X)$ works just as well for $S''(X)$, and one can easily see that the arguments before also imply that $$\beta \leq \liminf_{X \rightarrow \infty} \frac{|S''^{\zeta_3}(X)|}{|S''(X)|}$$ and $$\limsup_{X \rightarrow\infty} \frac{|S''^{\zeta_3}(X)|}{|S''(X)|} \leq \frac{3}{4}.$$

The same cannot be said of $S'(X)$, and there are separate issues for the upper and lower bounds.  For the lower bound, choose $\alpha \in S'(X)$.  Then one can write $\alpha = \pi_1^{a_1}\cdots\pi_n^{a_n}\zeta_3^a$, and write $\mathfrak{P}_i$ for the prime in $K(\sqrt[3]{\alpha})$ lying over $\pi_i$, and $\mathfrak{P}_{\lambda}$ for the prime lying over $\lambda$.  Then one sees that the relation in the class group is now $[\mathfrak{P}_1^{a_1} \cdots \mathfrak{P}_n^{a_n}] = [(1)]$.  The trick that we used to compute the Frobenius of $\mathfrak{P}_{\lambda}$ in $\Cl(K(\sqrt[3]{\alpha}))/(\sigma-1)$ no longer works.\\

Now write $V = \mathbb{F}_3\langle e_1, \ldots, e_{n-1}, e_{\lambda}\rangle$.  There is a map from $V \rightarrow \Cl(K(\sqrt[3]{\alpha}))[\sigma-1]$ given by sending $e_i \rightarrow \mathfrak{P}_i$ and $e_{\lambda} \rightarrow \mathfrak{P}_{\lambda}$.  Additionally, one can identify $V$ with the dual group $\Gal(K(\sqrt[3]{\pi_1}, \ldots, \sqrt[3]{\pi_{n-1}}, \sqrt[3]{\zeta_3}))^{\vee}$ by sending $e_i$ to the function $\sigma \rightarrow \frac{\sigma(\sqrt[3]{\pi_i})}{\sqrt[3]{\pi_i}}$ and $e_{\lambda}$ to $\sigma \rightarrow \frac{\sigma(\sqrt[3]{\zeta_3})}{\sqrt[3]{\zeta_3}}$.  Then one has that $\Gal(K(\sqrt[3]{\pi_1}, \ldots, \sqrt[3]{\pi_{n-1}}, \sqrt[3]{\zeta_3})) = \Cl(K(\sqrt[3]{\alpha}))/(\sigma-1)$, so we have identified $V$ with $\Cl(K(\sqrt[3]{\alpha}))^{\vee}[\sigma-1]$.  This lets us compute most of the matrix $M$ representing the pairing between the subgroup of $\Cl(L)[\sigma-1]$ generated by the primes $\mathfrak{P}_i$ and $\mathfrak{P}_{\lambda}$ with $\Cl(L)^\vee[\sigma-1]$; we can in particular compute $(e_i, e_j)$, $(e_i, e_{\lambda})$, and $(e_\lambda, e_j)$ as before, but we can no longer compute the diagonal entry $(e_{\lambda}, e_{\lambda})$ as we have no way of replacing $e_{\lambda}$ on the left with the $e_{i}$s.\\

To fix that issue, we can also identify $\Gal(K(\sqrt[3]{\pi_1}, \ldots, \sqrt[3]{\pi_{n-1}}, \sqrt[3]{\pi_n})) = \Cl(K(\sqrt[3]{\alpha}))/(\sigma-1)$.  Now, one has that the character $\chi$ of $\Gal(K(\sqrt[3]{\pi_1}, \ldots, \sqrt[3]{\pi_{n-1}}, \sqrt[3]{\pi_n}))$ given by sending $\sigma \rightarrow \frac{\sigma(\sqrt[3]{\pi_n})}{\sqrt[3]{\pi_n}}$ is equal to the image of $-\frac{a_1 e_1 + \cdots + a_{n-1}e_{n-1} + a e_{\lambda}}{a_n}$, so one gets that 
\[(e_{\lambda}, e_{\lambda}) = -\frac{a_{1}\log\left(\left(\frac{\lambda}{\pi_1}\right)_3\right) + \cdots + a_{1}\log\left(\left(\frac{\lambda}{\pi_n}\right)_3\right)}{a}.\]  

There is now a catch with this analysis: one has that $(e_i, e_{\lambda}) = 0$.  The matrix $M$ then has an $(n-1)$ by $(n-1)$ sub matrix $M'$ in the upper left hand part that is symmetric but still random by the previous discussion.  This matrix looks like this: $$
\mleft(
\begin{array}{ccc|c}
  \ast & \cdots & \ast & 0 \\
  \vdots & & \vdots & \vdots \\
  \ast & \cdots & \ast & 0 \\
  \hline
\ast & \cdots & \ast & \ast
\end{array}
\mright).
$$  The right-hand column of this matrix is all $0$s except the lower right hand corner.  Thus, $M$ has ful rank if and only if $M'$ has full rank and the coefficient $(e_{\lambda}, e_{\lambda})$ is nonzero.\\

This condition has a particularly simple form.  Write $\alpha' = \alpha\zeta_3^{-a}$.  Then $\alpha' \equiv 1 \pmod{\lambda^3}$; this coefficient is $0$ if and only if $\alpha' \equiv 1 \pmod{\lambda^4}$. Thus, one can easily adapt the arguments from before to show that $M$ has full rank $\frac{2\beta}{3}$ of the time. \\

The problem of the upper bound is much more difficult.  One can attempt to construct governing fields as before, but one is no longer able to show that what you want as the governing field is linearly disjoint with $K(\zeta_9)$ over $K$, so we can no longer compute degrees.  We currently see no way around this. \\

Putting this together, we get the following theorem:
\begin{theorem}\label{MainTheorem2}
One has the lower bound
\[\frac{91\beta}{93} \leq \liminf_{X \rightarrow \infty} \frac{|S^{\zeta_3}_{\lambda}(X)|}{|S_{\lambda}(X)|},\]
and one has the upper bound
\[\limsup_{X \rightarrow\infty} \frac{|S_{\lambda}^{\zeta_3}(X)|}{|S_{\lambda}(X)|} \leq \frac{95}{124}.\]
\end{theorem}

\end{document}